%% file: ERW_Finite_Drift.tex
\documentclass[reqno]{amsart}
\usepackage[utf8]{inputenc}
\usepackage{cite}
\usepackage{hyperref}
\usepackage{cleveref}
\usepackage{epigraph}
\usepackage[thicklines]{cancel}
\usepackage{xcolor}
\usepackage{theoremref}
\usepackage{bbm}

\usepackage{constants}
\newconstantfamily{constant}{symbol=C}

\input{preamble}
\allowdisplaybreaks

\newtheorem*{theorem*}{Theorem}
\usepackage[foot]{amsaddr}

\title[ERW with finite-drift cookie stacks]{Recurrence/Transience criteria for excited random walks with finite-drift cookie stacks}
\author[Zachary Letterhos]{Zachary Letterhos}
\date{9 March 2021}
\address{Department of Mathematics, Purdue University, 150 N University Street, West Lafayette, Indiana 47907, USA}
\keywords{Excited random walk; Cookie random walk; Recurrence; Transience; Branching-like processes}
\subjclass[2000]{Primary 60K35; Secondary 60J85}

\begin{document}
\maketitle
\begin{abstract}
    We consider excited random walk (ERW) on $\mathbb{Z}$ in environments with identical stacks of infinitely many cookies at each site, subject to the constraint that the total drift per site $\delta = \sum (2p_j - 1)$ is finite. Building on the methods of Kozma, Orenshtein, and Shinkar \cite{KOS_Periodic_Cookies}, we show that ERW in finite-drift environments is recurrent when $|\delta|<1$ and transient when $|\delta|>1$. In the case $|\delta|=1$ we prove that ERW is recurrent under mild assumptions on the environment. In addition, we show that ERW may be transient when $|\delta|=1$, an interesting new behavior that was not present in previously studied models. 
\end{abstract}

\section{Introduction}
Excited random walk (ERW) is a type of self-interacting nearest-neighbor random walk on $\mathbb{Z}$ in which the walker's next step from a site depends on the number of times the walker has previously visited that site. Formally, we define a \textit{cookie environment} $\left(\omega_{x}(j)\right)_{x \in \mathbb{Z}, j \geq 1} \in [0,1]^{\mathbb{Z}\times \mathbb{N}}$, and the excited random walk in this cookie environment is the stochastic process $\{X_n\}_{n \geq 0}$ such that $P(X_0 = 0) = 1$ and \begin{align*}
    P(X_{n+1} = X_n + 1 | X_0, X_1, \ldots, X_n) &= \omega_{X_n}(\#\{j \in \mathbb{N}: j \leq n, X_j = X_n\})\\
    P(X_{n+1} = X_n - 1 | X_0, X_1, \ldots, X_n) &= 1 - \omega_{X_n}(\#\{j \in \mathbb{N}: j \leq n, X_j = X_n\}).
\end{align*} The cookie terminology comes from a useful intuitive interpretation of ERW. First, a stack of cookies is placed at each site in $\mathbb{Z}$, and then a random walker is released. When the walker arrives at a site $x$ for the $j$th time, they consume the $j$th cookie in the stack, and this cookie induces a drift in the next step of the walk: the walker steps right with probability $\omega_x (j)$ and left with probability $1-\omega_x (j)$. We will refer to $\omega_x (j)$ as the \textit{strength} of the $j$th cookie at site $x$. If a cookie has strength $1/2$, and so doesn't induce any drift on the next step, we will call that cookie a \textit{placebo}. 

ERW was first introduced by Benjamini and Wilson in \cite{BenjWil}. The original model had only a single cookie at each site, or equivalently had $\omega_x (j) = 1/2$ for all $x \in \mathbb{Z}$ and for all $j > 1$. This model was extended by Zerner in \cite{ZernerMERW} to allow for multiple cookies at each site, and furthermore allowed for the cookie environments to be chosen according to a probability distribution. 

Zerner gave criteria for recurrence and transience of ERW in \cite{ZernerMERW} under the assumption of ``positive cookies,'' i.e. that $\omega_x (j) \geq 1/2$ for all $x$ and $j$. In this model, the long-term behavior of the walk is determined by $\delta$, the average drift contained in the cookies at each site:
\begin{align}
    \delta = \mathbb{E}\left[ \sum_{j=1}^{\infty}(2\omega_x (j) - 1)\right] \label{delta_definition}
\end{align}

In particular, Zerner proved that when $\delta > 1$, ERW is transient to $+\infty$. Basdevant and Singh \cite{BasdSingh} established that ERW is ballistic, i.e. has positive limiting speed, when $\delta > 2$ under the assumption of finitely many positive cookies at each site, or formally that there exists $M$ such that for every site $x$, $\omega_x (j) = 1/2$ whenever $j > M$. Their use of branching processes with migration laid the groundwork for much of the future work done with the ERW model. Kosygina and Zerner \cite{KosyZerner} removed the positive cookies assumption and proved recurrence/transience criteria for ERW in cookie environments with finitely many cookies at each site, namely that ERW in such environments is transient if and only if $|\delta|>1$. They also established a law of large numbers and annealed central limit theorem in the case that $|\delta|>4$. Functional limit laws for ERW with finitely many cookies at each site have also been proven, in the case $|\delta|\in(1,2)$ by Basdevant and Singh \cite{BasdevantSingh_Rate_of_Growth}, in the case $|\delta|\in(2,4)$ by Kosygina and Mountford \cite{KosyMount}, and in the case $|\delta|\leq1$ by Kosygina and Dolgopyat \cite{DolgKosy}. In all of these results, $\delta$ plays a key role: the criteria for recurrence/transience and ballisticity are given explicitly in terms of $\delta$, and the scaling exponents for the functional limit theorems directly depend on $\delta$. 

More general ERW with infinitely many cookies at each site have also been considered, but only with special structure. Kozma, Orenshtein, and Shinkar studied ERW with infinite cookie stacks with periodic structure \cite{KOS_Periodic_Cookies}, and Kosygina and Peterson worked with infinite cookie stacks generated by a Markov chain \cite{MarkovCookies}. These models have the additional complication that $\delta$, as expressed above in \eqref{delta_definition}, may not even be well-defined. Recurrence/transience criteria, law of large numbers, and functional limit theorems have been proven for both models \cite{Markov_Cookies_BMPE_KosyMountPeterson,FLL_Recurrent_ERW_Periodic,MarkovCookies}, but the results are given in terms of parameters that are associated to ERW through \textit{branching-like processes}, a generalization of the branching processes used for environments with finitely many cookies which we will discuss in a later section. These parameters can be computed explicitly, but their formulas are somewhat complicated.

\subsection{Description of model} We will study excited random walk in environments where deterministic, identical cookie stacks are placed at each site. That is, if we let $\boldsymbol{p} = (p_1 , p_2, p_3, \ldots)$ be a vector of cookie strengths, then $\omega_x (j) = p_j$ for all $x \in \mathbb{Z}$. As a slight abuse of terminology, we will refer to the environment with cookie stack $\boldsymbol{p}$ at each site as the cookie environment $\boldsymbol{p}$. We will also assume that $\boldsymbol{p}$ is an \textit{elliptic} cookie environment, i.e. that $p_i \in (0,1)$ for all $i \in \mathbb{N}$. We do not require that the cookies be positive or finite in number, and so this model is an extension of those studied in \cite{ZernerMERW} and \cite{KosyZerner} and contains those models as special cases. Since $\delta$ plays a key role in the results associated with the positive cookies and finitely many cookies per site models, it is natural to wonder what $\delta$ can tell us if we allow for infinitely many cookies at each site. However, because the environments we will consider can have a mix of positive and negative cookies, the total drift $\delta(\boldsymbol{p})$ at each site, defined by \begin{align*}
    \delta(\boldsymbol{p}) = \lim_{k \to \infty}\sum_{j=1}^k (2p_j - 1)
\end{align*}may not exist. In this paper we will only consider cookie environments where $\delta(\boldsymbol{p})$ exists and is finite. Of course, this assumption imposes structure on the cookie stacks $\boldsymbol{p}$ we can consider (for instance, we must have that $2p_j - 1 \to 0$, or equivalently $p_j \to 1/2$). As we will see, making this assumption still allows for rich and novel behavior of ERW.

For brevity, we will often write $\delta$ instead of $\delta(\boldsymbol{p})$. We will also frequently need to refer to the total drift contained in the first $m$ cookies, which we will denote $\delta_m$: \begin{align*}
    \delta_m = \sum_{j=1}^{m}(2p_j - 1).
\end{align*}

\subsection{Main results}
Our first main result is a criteria for recurrence/transience of ERW in $\boldsymbol{p}$. If $|\delta| \neq 1$, the long-term behavior of the walk is characterized by $\delta$. However, in the critical case $|\delta| = 1$, the rate at which the tail $|\sum_{j=n}^{\infty}(2p_j - 1)|$ decays to $0$ also plays a role. 
\begin{theorem}[Recurrence/transience]\thlabel{Recurrence_Transience} Let $\{X_n\}_{n \geq 0}$ be an excited random walk in a deterministic, identically-piled elliptic cookie environment $\boldsymbol{p}$ with finite total drift $\delta = \delta(\boldsymbol{p})$. \begin{enumerate}
    \item If $\delta > 1$, then $P\left(\lim_{n \to \infty} X_n = +\infty\right) = 1$,
    \item If $\delta < -1$, then $P\left(\lim_{n \to \infty}X_n = - \infty\right) = 1$,
    \item If $|\delta| < 1$, then $P\left(X_n = 0 \ \text{infinitely often}\right) = 1$,
    \item If $|\delta| = 1$ and it also holds that \begin{align}
        \left|\sum_{j=n}^{\infty}(2p_j - 1)\right| = o\left(\frac{1}{\log n}\right)\label{cookie_decay_assumption}
    \end{align} then $P\left(X_n = 0 \ \text{infinitely often}\right) = 1$.
\end{enumerate} 
\end{theorem} Note that \thref{Recurrence_Transience} does not describe the long term behavior of ERW in cookie environments where $|\delta|=1$, but the tail of the series that defines $\delta$ tends to $0$ more slowly than $\frac{1}{\log n}$. Because ERW is recurrent when $|\delta|=1$ under the assumption of positive or finitely many cookies at each site, it is natural to conjecture that ERW is \textit{always} recurrent when $|\delta|=1$. Although we do not prove that the condition in \eqref{cookie_decay_assumption} is necessary for recurrence, our next result shows that it is possible for ERW in $\boldsymbol{p}$ to be transient even when $|\delta|=1$.

\begin{theorem}\thlabel{Transient_example}There exist cookie environments $\boldsymbol{p}$ with $|\delta|=1$ where excited random walk is transient.
\end{theorem} 

The existence of such cookie environments distinguishes our model from those with only positive cookies or finitely many cookies per site. This behavior was also not present in previously-studied models that allowed for infinitely many positive and negative cookies at each site, such as the periodic and Markovian models. In those models, $\delta$ may be infinite or may not exist, and instead recurrence or transience can be deduced from three parameters of the associated branching-like process: $\mu$, $\theta$, and $\Tilde{\theta}$, which we will describe in more detail in the next section. For now, we will only remark that in the ``critical cases'' in the periodic and Markovian models which correspond to $|\delta|=1$ in the positive/finite cookie stack models, ERW is recurrent. We will prove \thref{Transient_example} by constructing an explicit environment $\boldsymbol{p}$ where $|\delta|=1$ but ERW is transient. 

\section{Associated branching-like processes}
\subsection{Forward branching-like processes}
Our main tool for proving \thref{Recurrence_Transience} will be Markov processes associated to the ERW in $\boldsymbol{p}$: the forward branching-like processes $\{Z_n^{+}\}_{n \geq 0}$ and $\{Z_n^{-}\}_{n \geq 0}$ that, essentially, track right and left excursions (respectively) of ERW from $0$. In fact, we can construct the forward branching processes and the ERW $\{X_n\}_{n \geq 0}$ using the same independent collection of Bernoulli random variables. To this end, for all $x \in \mathbb{Z}$ let $\xi_i^x \sim \mathrm{Ber}(p_i)$ be a sequence of independent random variables. We can then construct ERW using these Bernoulli random variables: \begin{align*}
    X_0 = 0, \quad \quad \quad X_{n+1} = X_n + \boldsymbol{1}_{\{\xi_{j}^{X_n} = 1\}} - \boldsymbol{1}_{\{\xi_{j}^{X_n} = 0\}} \ \text{ for }n > 0,
\end{align*}where $j = \#\{m \leq n : X_m = X_n\}$ corresponds to the number of times the walker has visited the present site. Intuitively, this construction corresponds to ``tossing all the coins at each site,'' then releasing a walker who starts at $0$ and moves according to the results of those coin tosses. To construct the forward branching-like process $\{Z_n^{+}\}_{n \geq 0}$ from the $\xi_i^x$, we define \begin{align*}
    T_m^x = \inf\{k \in \mathbb{N}:\sum_{i=1}^{k}(1-\xi_i^x) = m\}, \quad \quad \mathcal{S}^x(m) = \sum_{k=1}^{T_m^x} \xi_k^x.
\end{align*}In words, $T_m^x$ denotes the number of the Bernoulli trial on which the $m$th failure occurs at site $x$. Therefore, $\mathcal{S}^x(m)$ counts the number of successes in a sequence of $\mathrm{Ber}(p_i)$ trials before $m$ failures occur at site $x$. We then define $\{Z_n^{+}\}_{n \geq 0}$ to be the Markov process on $\mathbb{N}_0$ with transition probabilities given by\begin{align*}
    P(Z_n^{+} = k \ | \ Z_{n-1}^{+} = j) = P(\mathcal{S}^{n}(j) = k).
\end{align*}If we set $Z_0^{+} = 1$,  then $\{Z_n^{+}\}_{n \geq 0}$ tracks the first excursion by ERW to the right from $0$ in the following sense: if the walkers first step is to the right (i.e. $X_1 = 1$) and $H_{0} = \inf\{n > 0: X_n = 0\}$ is the time the ERW hits $0$, then on the event $\{X_1 = 1, \ H_{0}<\infty\}$, $Z_n^{+}$ is equal to the number of right steps the walker takes at site $n$ before $H_{0}$. On the event $\{X_1 = 1, \ H_0 = \infty\}$, $Z_n^{+}$ is stochastically greater than this number. For a more detailed discussion of this correspondence, see \cite{Zero_One_Law_Directional_Transience_1D_ERW}. Note, though, that $Z_n^{+} > 0$ for all $n \in \mathbb{N}$ if and only if $H_0 = \infty$. Since $P(H_0 = \infty)>0$ if and only if $P(X_n \to \infty)>0$, it follows that $P(Z_n^{+} >0 \text{ for all }n)>0$ if and only if $P(X_n \to \infty)>0$. 

In a similar way we define $\{Z_n^{-}\}_{n \geq 0}$, the forward branching-like process which tracks left excursions of ERW from $0$, except that we use an independent collection of $\mathrm{Ber}(1-p_i)$ random variables. To be concrete, we will use the collection $\{\Tilde{\xi_i^x}\}_{i \geq 1}$, where $\Tilde{\xi_i^x} \sim \mathrm{Ber}(1-p_i)$ random variables instead. By similar considerations to those above, we see that $P(X_n \to -\infty)>0$ if and only if $P(Z_n^{-}>0 \text{ for all }n)>0$. We can combine these observations with a zero-one law for directional transience of ERW that was proven by Amir, Berger, and Orenshtein in \cite{Zero_One_Law_Directional_Transience_1D_ERW} in order to connect the long-term behaviors of the ERW and the FBLP. We state their result as it pertains to our model, but note that it holds in a much more general setting.

\begin{theorem}[Theorem 1.2 of \cite{Zero_One_Law_Directional_Transience_1D_ERW}]\thlabel{Zero_One_Law_Directional_Transience_ERW} Let $\boldsymbol{p}$ be an elliptic, identically-piled cookie environment. Then \begin{align*}
    P\left(\lim_{n \to \infty}X_n = +\infty\right), \ P\left(\lim_{n \to \infty}X_n = -\infty\right) \in \{0,1\}.
\end{align*}
\end{theorem}

Together with the above considerations, \thref{Zero_One_Law_Directional_Transience_ERW} implies a connection between the survival of the FBLP and the recurrence/transience of ERW. In order to state it, we will require some additional notation. Let $P_n(\cdot) := P(\cdot|Z_0^{+} = n)$, let $\Tilde{P}_n(\cdot) := P(\cdot | Z_0^{-})$, and let $\mathbb{E}_n [\cdot]$ and $\Tilde{\mathbb{E}}[\cdot]$ denote the corresponding expectation.

\begin{theorem}[Theorem 2.3 of \cite{KOS_Periodic_Cookies}]\thlabel{FBLP_Walk_Connection} Let $\boldsymbol{p}$ be an elliptic cookie environment, and let $X$ be an ERW in $\boldsymbol{p}$. Then \begin{enumerate}
    \item $P_1(Z_n^{+} > 0 \text{ for all }n) > 0 \iff P(X_n \to +\infty) = 1$
    \item $P_1(Z_n^{-} > 0 \text{ for all }n) > 0 \iff P(X_n \to -\infty) = 1$
    \item $P_1(Z_n^{+} = 0 \text{ for some }n) = \Tilde{P}_1(Z_n^{-} = 0 \text{ for some }n) = 1 \iff P(X_n = 0 \ i.o. ) = 1$
\end{enumerate}
\end{theorem}For instance, in order to prove that ERW is transient to $+\infty$, \thref{FBLP_Walk_Connection} tells us that we can (equivalently) show that $Z_n^{+}$ ``does not die out'' with positive probability and that $Z_n^{-}$ ``dies out'' almost surely. \\

\noindent \textbf{Remark: }From this point on, we will often only be working with the collection $\{\xi_i^x\}_{i \geq 1}$ from a single site $x$ and the corresponding time of $m$th failure $T_m^x$. Since these random variables have the same distribution for each $x \in \mathbb{Z}$, we will drop the superscript $x$ as long as no confusion will arise from doing so. \\

We are now ready to analyze the connection between excited random walk and the forward branching-like processes. The parameters of the FBLP we are interested in are \begin{align*}
    \mu(n) &= \frac{\mathbb{E}_n[Z_1^{+}]}{n}  &&\rho(n)  = \mathbb{E}_n [Z_1^{+} - n] \\
    \nu(n) &= \frac{\mathrm{Var}(Z_1^{+} | Z_0 = n)}{n}   &&\theta(n)  = \frac{2\rho(n)}{\nu(n)}.
\end{align*}Furthermore, we will define \begin{align*}
    \mu &= \lim_{n \to \infty}\frac{\mathbb{E}_n[Z_1^{+}]}{n} &&\rho = \lim_{n \to \infty}\mathbb{E}_n [Z_1^{+} - n] \\
    \nu &= \lim_{n \to \infty}\frac{\mathrm{Var}(Z_1^{+} | Z_0 = n)}{n}   &&\theta  = \lim_{n \to \infty} \frac{2\rho(n)}{\nu(n)},
\end{align*}provided the limits exist. \\

\noindent \textbf{Remark: }The corresponding parameters for $Z_n^{-}$ will be denoted by $\Tilde{\mu}$, $\Tilde{\rho}$, $\Tilde{\nu}$, and $\Tilde{\theta}$, provided the limits exist. \\

We will now provide some intuition for the parameters of the FBLP. In the case of finitely many cookies per site, it has been shown by Kosygina and Mountford in \cite{KosyMount} that $Y_t^{(n)} = Z_{\lfloor nt \rfloor}^{+}/n$ can be approximated by a certain squared Bessel process. When the limits for the parameters of the FBLP exist, it would be natural to expect the same in our model. Given that $Y_t^{(n)} = y$, the increment $Y_{t + \Delta t} - Y_{t}$ is on average $\rho \Delta t$, with variance $\nu \cdot y$, and so we expect the scaling limit $Y_t$ to satisfy the SDE \begin{align*}
    dY_t &= \rho dt + \sqrt{2\nu} dB_t.
\end{align*} Analysis of this SDE reveals that the squared Bessel process that approximates $Y_{t}^{(n)}$ has generalized dimension $1 + \theta$. The approximation of scaled branching-like processes by squared Bessel processes has become a common technique in identifying the scaling limits of ERW, and has been used to identify scaling limits in the cases of finitely many cookies per site, periodic cookie stacks, and Markovian cookie stacks. We will not prove an SDE approximation for the FBLP in this paper since our results do not require it.

This connection to squared Bessel processes allows us to identify the ``critical cases'' in periodic and Markovian cookie environments mentioned beneath \thref{Transient_example}. A phase transition in the long-term behavior of squared Bessel processes occurs at dimension $2$: a $d$-dimensional squared Bessel process hits $0$ with probability $1$ if $d < 2$, but when $d \geq 2$ the probability is $0$. In light of that fact, it is not surprising that the case $\theta = 1$ requires special care. We will later see that in our model, $\theta = \delta$ and $\Tilde{\theta} = -\delta$, and so the ``critical cases'' occur when $|\delta|=1$.  In periodic and Markovian cookie environments, ERW is recurrent when $\theta = 1$, but in proving \thref{Transient_example} we will show that ERW in finite-drift cookie environments can be transient when $|\delta|=1$.

\section{Proof of \thref{Recurrence_Transience}}
Our main tool for proving \thref{Recurrence_Transience} will be the following theorem from \cite{KOS_Periodic_Cookies} that was used for establishing a recurrence/transience criteria in the case of periodic cookie stacks: 

\begin{theorem}[Theorem 1.3 of \cite{KOS_Periodic_Cookies}]\thlabel{Theorem_1.3_KOS_Periodic_Cookies}Let $\{Z_n^{+}\}_{n \geq 0}$ be the forward branching-like process associated to ERW in $\boldsymbol{p}$. Assume that the limit $\mu = \lim_{n \to \infty} \frac{\mathbb{E}_n [Z_1]}{n} = 1$, and that there is a constant $C$ such that for all $n$ sufficiently large and for all $\epsilon > 0$, it holds that \begin{align*}
    P_n\left(\left|\frac{Z_1}{n} - 1\right| > \epsilon \right) &\leq 2 \exp\left\{-\frac{C\epsilon^2 n}{2 + \epsilon}\right\}.
\end{align*} Then \begin{enumerate}
    \item[I.] if $\theta(n) < 1 + \frac{1}{\log(n)}-\alpha(n)n^{-1/2}$ for all sufficiently large $n$, where $\alpha(n)$ is such that $\alpha(n)\nu(n) \to \infty$ as $n \to \infty$, then $P(Z_n = 0 \text{ for some }n) = 1$,
    \item[II.] if $\theta(n) > 1 + \frac{2}{\log (n)} + \alpha(n) n^{-1/2}$ for all sufficiently large $n$, where $\alpha(n)$ is such that $\alpha(n)\nu(n) \to \infty$ as $n \to \infty$, then $P(Z_n > 0 \text{ for all }n) > 0.$
\end{enumerate}\end{theorem}
Before we proceed further, we will verify that the forward-branching like process is concentrated in the way that \thref{Theorem_1.3_KOS_Periodic_Cookies} requires.
\begin{theorem}[Concentration bound FBLP]\thlabel{Concentration_FBLP}
    Let $Z_n^{+}$ be the forward branching-like process associated to the ERW in cookie environment $\boldsymbol{p}$. Then \begin{align}
    P_n\left(\left|\frac{Z_1^{+}}{n} - 1\right|>\epsilon\right) &\leq 2 \exp \left\{-\frac{\Cl[constant]{Concentration_FBLP_constant}\epsilon^2 n}{2 + \epsilon}\right\}
\end{align}
\end{theorem}
\begin{proof}
First, we note that \begin{align*}
    P_n \left(Z_1^{+} > m\right) = P\left(\sum_{i=1}^{m+n}\xi_i > m\right),
\end{align*}where $\{\xi_i\}_{i \geq 1}$ is an independent sequence of Bernoulli random variables with $\xi_i \sim \mathrm{Ber}(p_i)$. By centering appropriately, we see that \begin{align*}
    P\left(\sum_{i=1}^{m+n}\xi_i > m\right) &= P\left(\sum_{i=1}^{m+n}(\xi_i - p_i) > m - \sum_{i=1}^{m+n}p_i \right) \\
    &= P\left(\sum_{i=1}^{m+n}(\xi_i - p_i) > \frac{m-n}{2} - \frac{1}{2}\sum_{i=1}^{m+n}(2p_i - 1)\right)\\
    &= P\left(\sum_{i=1}^{m+n}(\xi_i - p_i) > \frac{m-n-\delta_{m+n}}{2}\right)
\end{align*}Note that when $m>n$ and is sufficiently large, $(m-n-\delta_{n+m})/2 > 0$ because $\{\delta_k\}_{k \geq 1}$ is bounded. Applying Hoeffding's inequality to the expression above gives \begin{align}
    P_n (Z_1^{+} > m) &= P\left(\sum_{i=1}^{m+n}(\xi_i - p_i) > \frac{m-n-\delta_{m+n}}{2}\right) \nonumber\\ 
    &\leq \exp\left\{- \frac{\left((m-n) - \delta_{m+n}\right)^2}{2(m+n)}\right\} \nonumber\\
    &= \exp\left\{-\frac{(m-n)^2 - 2(m-n)\delta_{m+n} + \delta_{m+n}^2}{2(m+n)}\right\} \nonumber\\
    &\leq \exp\left\{-\frac{(m-n)^2}{2(m+n)}\right\}\exp\left\{\frac{(m-n)\delta_{m+n}}{m+n}\right\} \label{Concentration_Hoeffding_application}
\end{align}provided $m$ is sufficiently large. A corresponding bound on $P_n (Z_1^{+} < m)$ can be obtained in a similar way. We can use the bound in \eqref{Concentration_Hoeffding_application} to obtain the claimed concentration inequality: \begin{align*}
    P_n\left(\left|\frac{Z_1^{+}}{n} - 1\right|>\epsilon\right) &= P_n \left(\left|Z_n^{+} - n\right|>n\epsilon\right) \\
    &= P_n (Z_n^{+} > n(1+\epsilon)) + P_n (Z_n^{+} < n(1-\epsilon) ) \\
    &\leq P_n (Z_n^{+} > \lfloor n(1+\epsilon)\rfloor) + P_n (Z_n^{+} < \lceil n(1-\epsilon)\rceil) \\
    &\leq 2\exp\left\{-\frac{\Cl[constant]{intermediate_1} n^2 \epsilon^2}{(2+\epsilon)n}\right\}\exp\left\{\frac{\Cl[constant]{intermediate_2} \epsilon \delta_{\lfloor(2+\epsilon)n\rfloor}}{2+\epsilon}\right\} \\
    &\leq 2\exp\left\{-\frac{\Cl[constant]{Hoeffding_bound_constant} n \epsilon^2}{2+\epsilon}\right\},
\end{align*}where $\Cr{Hoeffding_bound_constant}$ depends only on the sequence $\{\delta_k\}_{k \geq 1}$.
\end{proof} Note that since, conditional on $\{Z_0^{+} = n\}$, $T_n = Z_1^{+}+ n$, \thref{Concentration_FBLP}  also gives a corresponding concentration bound on $T_n$: \begin{align}
    P\left(\left|\frac{T_n}{n} - 2\right| > \epsilon\right) \leq 2\exp\left\{-\frac{C \epsilon^2 n}{2 + \epsilon}\right\}\label{T_n_concentration}
\end{align}

Since we have verified the concentration requirement of \thref{Theorem_1.3_KOS_Periodic_Cookies}, we now only need to compute parameters of the FBLP and bound their convergence rates.

\subsection{Parameters of forward branching-like process} Recall the parameters of the forward branching-like process that we are interested in: \begin{align*}
    \mu &= \lim_{n \to \infty}\frac{\mathbb{E}_n[Z_1^{+}]}{n}  &&\rho(n)  = \mathbb{E}_n [Z_1^{+} - n] \\
    \nu(n) &= \frac{\mathrm{Var}(Z_1^{+} | Z_0^{+} = n)}{n}   &&\theta(n)  = \frac{2\rho(n)}{\nu(n)}.
\end{align*}

\noindent \textbf{Remark: }By symmetry, $\{Z_n^{-}\}_{n \geq 0}$ for ERW is the environment $\boldsymbol{p}$ follows the same distribution as $\{Z_n^{+}\}_{n \geq 0}$ in the reflected environment $\boldsymbol{1 - p} = (1-p_1 , 1-p_2 , \ldots)$. Therefore, if we prove a result for $\{Z_n^{+}\}_{n \geq 0}$, we can use symmetry to deduce a corresponding result for $\{Z_n^{-}\}_{n \geq 0}$. Our next objective is to compute the parameters associated to the FLBP and determine the rates of convergence for $\rho$, $\nu$, and $\theta$.  

\begin{theorem}[Parameters of FBLP]\thlabel{Parameters_FBLP} Let $\boldsymbol{p}$ be a cookie environment with finite total drift $\delta(\boldsymbol{p})$, and let $\{Z_n\}_{n \geq 0}$ be the forward branching-like process associated to the excited random walk in this environment. Then \begin{enumerate}
    \item We have $\mu = 1$, $\rho = \delta(\boldsymbol{p})$, $\nu = 2$, and $\theta = \delta(\boldsymbol{p})$. 
    \item We have $|\nu(n) - 2| = \mathcal{O}\left(n^{-1/2}\log^4 n + b_n \log^4 b_n\right)$, where $b_n = \frac{1}{4n}\sum_{j=1}^n (2p_j - 1)^2$. Under the assumption in \eqref{cookie_decay_assumption}, we have that $|\nu(n) - 2| = o\left(\frac{1}{\log n}\right)$.
    \item Under the assumption in \eqref{cookie_decay_assumption}, $|\rho(n) - \delta| = o\left(\frac{1}{\log n}\right)$.
    \item Under the assumption in \eqref{cookie_decay_assumption}, $|\theta(n) - \delta| = o\left(\frac{1}{\log n}\right)$. 
\end{enumerate}
\end{theorem}
Before proving this statement, we will show how to prove \thref{Recurrence_Transience} by combining \thref{Theorem_1.3_KOS_Periodic_Cookies} and \thref{Parameters_FBLP}. \\

\noindent \textit{Proof of \thref{Recurrence_Transience} assuming \thref{Parameters_FBLP}.} Let $\boldsymbol{p}$ be a cookie environment with finite total drift $\delta(\boldsymbol{p})$. Let $\{Z^{+}_n\}_{n \geq 0}$ and $\{Z_n^{-}\}_{n \geq 0}$ be the forward branching-like processes associated to the ERW in $\boldsymbol{p}$.

Intuitively, while the branching-like process $Z^{+}$ is positive the ERW takes an excursion to the right, and when the branching-like process dies out the ERW returns to $0$. In a similar way, $Z^{-}$ tracks left excursions of the ERW (by symmetry, a left excursion in the cookie environment $\boldsymbol{p}$ corresponds to a right excursion in the ``reflected cookie environment'' $\boldsymbol{1-p}$). We showed in \thref{Concentration_FBLP} that both the branching-like processes are concentrated in the way that we need to apply \thref{Theorem_1.3_KOS_Periodic_Cookies}, and a quick calculation shows that $\delta(\boldsymbol{1-p}) = -\delta(\boldsymbol{p})$. We consider each of the cases described in \thref{Recurrence_Transience} in turn:
\begin{itemize}
    \item Suppose that $\delta(\boldsymbol{p}) > 1$, and therefore $\delta(\boldsymbol{1-p}) < -1$. By \thref{Theorem_1.3_KOS_Periodic_Cookies}, $P_1(Z^{+}_n > 0 \text{ for all }n) > 0$ and $\Tilde{P}_1(Z^{-}_n = 0 \text{ for some }n) = 1$. Therefore the ERW in $\boldsymbol{p}$ is transient to $+\infty$ a.s. by \thref{FBLP_Walk_Connection}.
    
    \item Similarly, when $\delta(\boldsymbol{p}) < -1$, we have $\delta(\boldsymbol{1-p}) > 1$. It then follows from \thref{Theorem_1.3_KOS_Periodic_Cookies} that we have $P_1(Z^{+}_n = 0 \text{ for some }n) = 1$ and $\Tilde{P}_1(Z^{-}_n > 0 \text{ for all }n)>0.$ Therefore ERW in $\boldsymbol{p}$ is transient to $-\infty$ a.s. in this case.
    
    \item In the case that $\delta(\boldsymbol{p}) \in (-1,1)$, we also have that $\delta(\boldsymbol{1-p}) \in (-1,1)$. By \thref{Theorem_1.3_KOS_Periodic_Cookies}, we see that $P_1(Z^{+}_n = 0 \text{ for some }n) = \Tilde{P}_1(Z^{-}_n = 0 \text{ for some }n) = 1$, and so in this case the ERW in $\boldsymbol{p}$ is recurrent a.s.
    
    \item In the critical case where $\delta = \pm 1$, extra care is needed. In this case, we must rely on the fact that the assumption in (\ref{cookie_decay_assumption}) is sufficient to guarantee that part I. of \thref{Theorem_1.3_KOS_Periodic_Cookies} is satisfied. Under that assumption, we can apply the theorem to show that the ERW in $\boldsymbol{p}$ is recurrent a.s. when $\delta = \pm 1$. \hfill $\blacksquare$\\
    \end{itemize}

We will now prove \thref{Parameters_FBLP}. In order to do so, we must compute the parameters of the forward branching-like process and determine their rates of convergence.

\subsubsection{Computing $\rho$ and $\mu$}
Recall that, conditional on $\{Z_0^{+} = n\}$, we can express $Z_1^{+}$ as \begin{align*}
    Z_1^{+} = \sum_{k=1}^{T_n}\xi_k,
\end{align*}where $T_n$ is the trial on which the $n$th failure in the sequence of $\mathrm{Ber}(p_i)$ trials that determine the value of $Z_1^{+}$ occurs. To compute $\rho$ and $\mu$, we will relate $T_n$ and $Z_1^{+}$ in two different ways. First, we can use a version of Wald's identity (see \thref{Wald} in the appendix) to write \begin{align*}
    \mathbb{E}_n[Z_1^{+}] &= \mathbb{E} \left[\sum_{k=1}^{T_n} p_k\right] = \mathbb{E}\left[\frac{1}{2}\sum_{k=1}^{T_n}(2p_k -1) + \frac{1}{2}T_n\right] = \frac{1}{2}\mathbb{E}[\delta_{T_n}] + \frac{1}{2}\mathbb{E}[T_n].
\end{align*}On the other hand, since conditional on $\{Z_0^{+} = n\}$, $T_n = Z_1^{+}+n$, we also have \begin{align*}
    \mathbb{E}[T_n] &= \mathbb{E}_n [Z_1^{+}] + n.
\end{align*}Substituting this expression for $\mathbb{E}[T_n]$ back into the first expression for $\mathbb{E}_n [Z_1^{+}]$, we obtain $\mathbb{E}_n[Z_1^{+}] = n + \mathbb{E}[\delta_{T_n}]$. We can then easily compute $\rho$: \begin{align}
    \rho = \lim_{n \to \infty} \mathbb{E}_n [Z_1^{+} - n] = \lim_{n \to \infty} \mathbb{E}[\delta_{T_n}] = \delta, \label{compute_rho}
\end{align}where the final equality follows from the fact that $T_n \geq n$, that the sequence $\{\delta_k\}_{k \geq 1}$ is convergent (hence bounded), and the dominated convergence theorem. In addition, the above discussion also implies that $\mu = \lim_{n \to \infty} \mathbb{E}_n [Z_1]/n = 1$. Finally, we can obtain a bound on the rate of convergence of $\rho(n)$ to $\rho$ under the assumption in \eqref{cookie_decay_assumption}: \begin{align*}
    \left|\rho(n) - \rho\right| =  \left|\mathbb{E}\left[\delta_{T_n}\right] - \delta\right| = \left|\mathbb{E}\left[\sum_{j= T_{n}+1}^{\infty}(2p_j - 1)\right]\right| \leq \sup_{m \geq n} \left|\sum_{j=m+1}^{\infty}(2p_j - 1)\right| = o\left(\frac{1}{\log n}\right).
\end{align*}

\subsubsection{Computing $\nu$}
Our primary tool for computing $\nu$ is a modified version of Lemma $2.5$ from the periodic cookies setting in \cite{KOS_Periodic_Cookies}. Before stating it, we must define two additional parameters associated to the cookie environment $\boldsymbol{p}$.
\begin{lemma} \thlabel{nu_calculation_lemma} Let $\boldsymbol{p}$ be a deterministic, identically-piled elliptic cookie environment such that $\delta(\boldsymbol{p})$ exists and is finite. We define $\Bar{p}_n = \frac{1}{n}\sum_{j=1}^n p_j$ and $A_n = \frac{1}{n}\sum_{j=1}^n p_j (1-p_{j})$. Then \begin{enumerate}
    \item[(a)] $\Bar{p}_n \to \frac{1}{2}$ as $n \to \infty$, and there exists a constant $\Cl[constant]{nu_calculation_constant}$ such that $|\Bar{p}_n - \frac{1}{2}| \leq \frac{\Cr{nu_calculation_constant}}{n}$ for all $n \in \mathbb{N}$.
    \item[(b)] $A_n \to \frac{1}{4}$ as $n \to \infty$.
\end{enumerate}
    
\end{lemma}It should be noted that both of the statements in \thref{nu_calculation_lemma} were part of the hypotheses of Lemma 2.5 in the periodic cookie setting \cite{KOS_Periodic_Cookies}, with part (b) requiring only that $\lim A_n$ exists and is strictly positive. For cookie environments where $\delta(\boldsymbol{p})$ exists and is finite, it turns out that $\Bar{p}_n$ always tends to $\frac{1}{2}$ with $\mathcal{O}\left(\frac{1}{n}\right)$ error and $A_n$ always tends to $\frac{1}{4}$.
\begin{lemma}[Modified from Lemma 2.5 of \cite{KOS_Periodic_Cookies}]\thlabel{Periodic_cookies_nu_calculation} Let $\boldsymbol{p}$ be a deterministic, identically-piled elliptic cookie environment such that $\delta(\boldsymbol{p})$ exists and is finite, and let $b_n = |A_n - \frac{1}{4}|$.
Then \begin{align*}
    \nu = \lim_{n \to \infty}\nu(n) = \lim_{n \to \infty}\frac{1}{n}\mathrm{Var}(Z_1^{+}| Z_0 = n) = 2.
\end{align*}In addition, we can give a bound on the rate of convergence: \begin{align*}
    |\nu(n) - 2| = \left|\frac{1}{n}\mathrm{Var}(Z_1^{+}|Z_0 = n) - 2 \right| = \mathcal{O}\left(\frac{\log^4 n}{\sqrt{n}} + b_n \log^4 b_n\right).
\end{align*}
where the constant implied by the $\mathcal{O}(\cdot)$ notation depend only on the cookie environment $\boldsymbol{p}$.
\end{lemma}
\noindent \textbf{Remark: }Note that $n^{-1/2}\log^4 n$ is the dominant term when $b_n \ll n^{-1/2}$, but that $b_n \log^4 b_n$ is dominant when $b_n \gg n^{-1/2}$.


\begin{proof}[Proof of \thref{nu_calculation_lemma}]
To prove part (a) of \thref{nu_calculation_lemma}, note that \begin{align*}
    \left|\Bar{p}_n - \frac{1}{2}\right| = \frac{1}{n}\left|\sum_{j=1}^n \left(p_j - \frac{1}{2}\right)\right| = \frac{|\delta_n |}{2n}.
\end{align*}Since $\{\delta_n \}_{n \geq 1}$ is bounded, there exists a constant $\Cr{nu_calculation_constant}$ such that $|\delta_n | / 2n \leq \frac{\Cr{nu_calculation_constant}}{n}$ for all $n \in \mathbb{N}$. To prove (b), observe that \begin{align*}
    \left|A_n - \frac{1}{4}\right| = \left|\frac{1}{n}\sum_{j=1}^{n}p_j(1-p_j) - \frac{1}{4}\right| = \frac{1}{n}\left|\sum_{j=1}^n \left(p_j - \frac{1}{2}\right)\left(\frac{1}{2} - p_j\right)\right| = \frac{1}{4n}\left|\sum_{j=1}^{n}(2p_j - 1)^2\right|. 
\end{align*}Since we assume the series which defines $\delta$ converges, $(2p_j - 1)^2 \to 0$ as $j \to \infty$. Then $\sum_{j=1}^n (2p_j - 1)^2 = o(n)$, and so $A_n \to \frac{1}{4}$ as $n \to \infty$. \end{proof}
\begin{proof}[Proof of \thref{Periodic_cookies_nu_calculation}]
The proof is the same in spirit as the proof of Lemma 2.5 in \cite{KOS_Periodic_Cookies}. The key difference comes in obtaining the rate of convergence for $\frac{1}{n}\mathbb{E}_n\left[(Z_1^{+}- n)^2\right]$, and we will point out adjustments that need to be made along the way. First, we can use the fact that $\mathbb{E}_n [Z_1^{+}] = n + \mathbb{E}[\delta_{T_n}]$ to see that  \begin{align*}
    \mathrm{Var}(Z_1^{+}| Z_0^{+} = n) &= \mathbb{E}\left[(Z_1^{+}- n)^2\right] - \mathbb{E}[\delta_{T_n}]^{2},
\end{align*}and so there exists a constant $C$ such that $\frac{1}{n}\left|\mathrm{Var}(Z_1^{+}| Z_0 = n) - \mathbb{E}[(Z_1^{+}- n)^2]\right| \leq \frac{C}{n}$. It will therefore suffice to work with $\mathbb{E}[(Z_1^{+}- n)^2]$. We will start by rewriting this quantity: \begin{align*}
    \mathbb{E}_n \left[(Z_1^{+}- n)^2\right] &= \sum_{t=0}^{\infty}(2t+1)P_n\left(|Z_1^{+}- n| > t\right) \\
    &= 2\sum_{t=0}^{\infty}t \cdot P_n \left(|Z_1^{+}- n| > t\right) + \sum_{t = 0}^{\infty} P_n (|Z_1^{+}- n|) > t \\
    &= 2 \sum_{t = 0}^{\infty}t \cdot P_n(|Z_1^{+}- n|>t) + \mathbb{E}_n\left[|Z_1^{+}- n|\right].
\end{align*}Due to the concentration bound in \thref{Concentration_FBLP}, $\mathbb{E}_n[|Z_1^{+}- n|] = \mathcal{O}(\sqrt{n})$, and so $\frac{1}{n}\mathbb{E}_n[|Z_1^{+}- n|] = \mathcal{O}(n^{-1/2})$. Therefore, to prove our claim we only need to bound \begin{align*}
    \left|\frac{1}{n}\sum_{t=0}^{\infty}t \cdot P_n (|Z_1^{+}- n|>t) - 1\right|.
\end{align*}In fact, we will show that for any sequence $a = a_n$ such that $a_n \stackrel{n \to \infty}{\longrightarrow}\infty$ and $a_n \leq \sqrt{n}$ (for $n$ sufficiently large) that \begin{align*}
    \left|\frac{1}{n}\sum_{t=0}^{\infty}t \cdot P_n (|Z_1^{+}- n|>t) - 1\right| = \mathcal{O}\left(a^4 n^{-1/2} + \left|A_n - \frac{1}{4}\right|a^2 + e^{-C' a}\right),
\end{align*} then we will choose an appropriate sequence $a_n$ to obtain the conclusion. To obtain the bound we need, we will rewrite the events involving $Z_1^{+}$ in terms of the number of failures that occur in the coin tosses that determine the value of $Z_1^{+}$. Let $F_n$ be the number of failures in this sequence of coin tosses after the $n$th toss: \begin{align*}
    F_n := \sum_{j=1}^n \left(1 - \xi_j\right).
\end{align*}Then we have that \begin{align*}
    \left|\frac{1}{n}\sum_{t=0}^{\infty}t \cdot P_n(|Z_1^{+}- n| > t) - 1\right| &= \left|\frac{1}{n}\sum_{t=0}^{\infty}t\left(P_n(Z_1^{+}> n+t) + P(Z_1^{+}< n-t)\right) - 1\right| \\
    &= \left|\frac{1}{n}\sum_{t=0}^{\infty}t \cdot \left(P(F_{2n+t} < n) +P(F_{2n-t-1} \geq n)\right) - 1\right|.
\end{align*}As in the periodic cookies case, we divide the sum into its ``head'' and ``tail.'' Let \begin{align*}
    H_n (a) &= \sum_{t=0}^{\lfloor a\sqrt{n}\rfloor} t \cdot \left(P(F_{2n+t} < n) + P(F_{2n-t-1} \geq n)\right)\\
    T_n (a) &= \sum_{t = \lfloor a \sqrt{n} \rfloor}^{\infty} t \cdot \left(P(F_{2n+t} < n) + P(F_{2n-t-1} \geq n)\right),
\end{align*}where $a = a_n$ can be any sequence which grows slowly with $n$ (again, we will specify an appropriate choice of $a_n$ at the end of the proof). We can handle the tail $T_n (a)$ by using the following result from the periodic cookies setting. \begin{lemma}[Claim 2.8 of \cite{KOS_Periodic_Cookies}]\thlabel{Tail_sum_bound} Let $a = a_n$ such that $\lim_{n \to \infty}a_n = \infty$. Then for all sufficiently large $n \in \mathbb{N}$, \begin{align*}
    \frac{1}{n}T_n (a) \leq \Cl[constant]{Tail_sum_constant1} e^{-\Cl[constant]{Tail_sum_constant2}a}.
\end{align*}
\end{lemma} \noindent The proof of this result in our case is identical to the one in \cite{KOS_Periodic_Cookies}, and so we will omit it. We now turn our attention to the head of the sum $H_n (a)$. The main thrust of the argument is that $H_n (a)$ can be approximated by a sum over the standard normal cdf $\Phi$.
\begin{lemma}[Claim 2.6 of \cite{KOS_Periodic_Cookies}] \thlabel{Head_sum_gaussian_approx} Let $a > 0$ and let $n \in \mathbb{N}$ be such that $a = a_n \leq \sqrt{n}$. Then \begin{align*}
    \left|\frac{1}{n}H_n (a) - \frac{1}{n}\sum_{t = 0}^{\lfloor a \sqrt{n}\rfloor}2t \cdot \Phi\left(\frac{-t}{\sqrt{8 A_n}}\right)\right| \leq \Cl[constant]{Head_sum_gaussian_approx_constant}\left(\frac{a^{4}}{\sqrt{n}} + \left|A_n - \frac{1}{4}\right|a^2\right)
\end{align*}
\end{lemma}
\begin{lemma}[Claim 2.7 of \cite{KOS_Periodic_Cookies}]\thlabel{Head_sum_approx_error} Let $a = a_n$ such that $\lim_{n\to\infty}a_n = \infty$. Then \begin{align*}
    \lim_{n \to \infty} \frac{1}{n}\sum_{t=0}^{\lfloor a \sqrt{n}\rfloor} 2t \cdot \Phi\left(\frac{-t}{\sqrt{8 A_n}}\right) = 1,
\end{align*}where $\Phi$ is the standard normal cdf. Moreover, \begin{align*}
    \frac{1}{n}\sum_{t=0}^{\lfloor a \sqrt{n}\rfloor} 2t \cdot \Phi\left(\frac{-t}{\sqrt{8 A_n}}\right) = 1 + \mathcal{O}\left(\frac{a}{\sqrt{n}} + \exp\left(-Ca\right)\right)
\end{align*}
\end{lemma} \noindent The proof is of \thref{Head_sum_approx_error} is identical to the proof of Claim 2.7 in \cite{KOS_Periodic_Cookies}, and so we will not repeat it. \\

\noindent \textit{Proof of \thref{Head_sum_gaussian_approx}.} Let $\Bar{q}_n = \frac{1}{n}\sum_{j=1}^{n}(1 - p_j )$, let $\sigma_j^2 = \mathbb{E}\left[\left((1 - \xi_j) - (1 - p_j)\right)^2\right] = p_j (1 - p_j)$, and let $\rho_j = \mathbb{E}\left[\left|\big((1 - \xi_j) - (1 - p_j)\big)^3\right|\right]$ for all $j \in \mathbb{N}$. By the Berry-Esseen theorem, there is a constant $\Cl[constant]{Berry_Esseen_constant}$ such that for all $\alpha \in \mathbb{R}$, \begin{align*}
    \left|P\left(\frac{F_n - n \Bar{q}_n }{\sqrt{n A_n}} \leq \alpha\right) - \Phi(\alpha)\right| \leq \Cr{Berry_Esseen_constant} \cdot \left(\sum_{i=1}^{n}\sigma_i^2\right)^{-3/2} \cdot \left(\sum_{i=1}^{n}\rho_i\right),
\end{align*}where $\Phi$ is the standard normal cdf. Since $\rho_i \leq 1$ for each $i$ and $A_n$ is bounded, we have that \begin{align*}
    \left(\sum_{i=1}^{n}\sigma_i^2\right)^{-3/2} \cdot \left(\sum_{i=1}^{n}\rho_i\right) \leq (n A_n )^{-3/2} \cdot n \leq C n^{-1/2}.
\end{align*}Therefore \begin{align}
    P\left(\frac{F_n - n \Bar{q}_n }{\sqrt{n A_n}} \leq \alpha\right) = \Phi(\alpha) + \mathcal{O}(n^{-1/2}). \label{Head_sum_gaussian_approx_key_line}
\end{align}We will now prove that we can replace $A_n$ and $\Bar{q}_n$ by their respective limits $\frac{1}{4}$ and $\frac{1}{2}$. The following lemma shows that we can do just that, and is analogous to Claim 2.9 in \cite{KOS_Periodic_Cookies}.
\begin{lemma}\thlabel{KOS_key_difference} \begin{align*}
    \left|P\left(\frac{F_n - n \Bar{q}_n}{\sqrt{n A_n}} \leq \alpha\right) - P\left(\frac{F_n - \frac{1}{2}n}{\sqrt{n(1/4)}} \leq \alpha\right)\right| \leq \Cl[constant]{KOS_key_difference_constant1}n^{-1/2} + \Cl[constant]{KOS_key_difference_constant2}\left|A_n - \frac{1}{4}\right|.
\end{align*}
\end{lemma}
\noindent \textit{Proof of \thref{KOS_key_difference}.} The proof follows the same technique as the proof of Claim 2.9 in \cite{KOS_Periodic_Cookies}, but we must keep track of the error in terms of $b_n = |A_n - 1/4|$: \begin{align*}
    \left|P\left(\frac{F_n - n \Bar{q}_n}{\sqrt{n A_n}} \leq \alpha\right) - P\left(\frac{F_n - \frac{1}{2}n}{\sqrt{n(1/4)}} \leq \alpha\right)\right| &= \left|P\left(\frac{F_n - n \Bar{q}_n}{\sqrt{n A_n}} \leq \alpha\right) - P\left(\frac{F_n - n \Bar{q}_n}{\sqrt{n A_n}} \leq \sqrt{\frac{(1/4)}{A_n}}\alpha + \frac{n\left(\frac{1}{2} - \Bar{q}_n\right)}{\sqrt{n A_n}}\right)\right|\\
    &\leq \left|\Phi(\alpha) - \Phi\left(\sqrt{\frac{(1/4)}{A_n}}\alpha + \frac{n\left(\frac{1}{2} - \Bar{q}_n\right)}{\sqrt{n A_n}}\right)\right| + Cn^{-1/2},
\end{align*}where this last inequality follows from \eqref{Head_sum_gaussian_approx_key_line}. Using the fact that $\Phi$ is $\frac{1}{\sqrt{2\pi}}-$Lipschitz, we can bound this expression by \begin{align*}
    \frac{1}{\sqrt{2 \pi}}\left|\alpha - \left(\sqrt{\frac{(1/4)}{A_n}}\alpha + \frac{n\left(\frac{1}{2} - \Bar{q}_n\right)}{\sqrt{n A_n}}\right)\right| + Cn^{-1/2} &\leq \frac{1}{\sqrt{2 \pi}}\left(\alpha\left|1 - \sqrt{\frac{(1/4)}{A_n}}\right|+\frac{n\left|\frac{1}{2} - \Bar{q}_n\right|}{\sqrt{n A_n}} \right) + Cn^{-1/2}. \\
\end{align*}By \thref{nu_calculation_lemma}, $\Bar{q}_n = \frac{1}{2} + \mathcal{O}(1/n)$, and so we can bound $n\left(\frac{1}{2} - \Bar{q}_n\right)$ by a constant to obtain
\begin{align*}
    \frac{1}{\sqrt{2 \pi}}\left(\alpha\left|1 - \sqrt{\frac{(1/4)}{A_n}}\right|+\frac{n\left|\frac{1}{2} - \Bar{q}_n\right|}{\sqrt{n A_n}} \right) + Cn^{-1/2} &\leq \frac{1}{\sqrt{2 \pi}}\left(\alpha\left|1 - \sqrt{\frac{(1/4)}{A_n}}\right| + \frac{\Cr{nu_calculation_constant}}{\sqrt{n A_n}}\right) + Cn^{-1/2} \\
    &\stackrel{*}{=} \frac{1}{\sqrt{2\pi}}\left(\alpha \frac{\left|A_n - \frac{1}{4}\right|}{|\sqrt{A_n}\left(\sqrt{A_n} + \frac{1}{4}\right)|} + \frac{\Cr{nu_calculation_constant}}{\sqrt{n A_n}}\right) + Cn^{-1/2},
\end{align*}where the starred equality follows from multiplying the numerator and denominator of the first term by $\sqrt{A_n} + \sqrt{\frac{1}{4}}$. Recall that $A_n \to \frac{1}{4}$, and moreover $A_n = \frac{1}{n}\sum_{j=1}^n p_j (1-p_j ) \geq \left(\min_{j} p_j (1-p_j )\right) > 0$ because we assume that $p_j \to \frac{1}{2}$ and that $p_j \notin \{0,1\}$ for all $j \in \mathbb{N}$. Therefore the term above is $\mathcal{O}\left(|A_n - \frac{1}{4}| + n^{-1/2}\right)$, and so we have established that
\begin{align*}
\frac{1}{\sqrt{2 \pi}}\left|\alpha - \left(\sqrt{\frac{A}{A_n}}\alpha + \frac{n\left(\frac{1}{2} - \Bar{q}_n\right)}{\sqrt{n A_n}}\right)\right| + Cn^{-1/2} &\leq \Cr{KOS_key_difference_constant1} n^{-1/2} + \Cr{KOS_key_difference_constant2}\left|A_n - \frac{1}{4}\right|,
\end{align*}which completes the proof of \thref{KOS_key_difference}. \hfill $\blacksquare$ \\

We now have the tools that we need to finish the proof of \thref{Head_sum_gaussian_approx}. The remaining proof follows the same structure as in \cite{KOS_Periodic_Cookies}, although we must keep track of additional error terms. By \thref{KOS_key_difference}, we have that \begin{align*}
    H_n (a) &= \sum_{t=0}^{\lfloor a \sqrt{n} \rfloor} t \cdot \left[ P\left(\frac{F_{2n + t} - \frac{2n + t}{2}}{\sqrt{(2n+t)(1/4)}} < \frac{-t}{2 \sqrt{(2n+t)(1/4)}}\right) + P\left(\frac{F_{2n-t-1} - \frac{2n-t-1}{2}}{\sqrt{(2n-t-1)(1/4)}} \geq \frac{t+1}{2\sqrt{2n-t-1 (1/4)}}\right)\right] \\
    &= \sum_{t=0}^{\lfloor a \sqrt{n} \rfloor} t \left[\Phi\left(\frac{-t}{2\sqrt{(2n+t)(1/4)}}\right) + 1 - \Phi\left(\frac{t+1}{2\sqrt{(2n-t-1)(1/4)}}\right) + \mathcal{O}\left(n^{-1/2} + \left|A_n - \frac{1}{4}\right|\right)\right].
\end{align*} Again using the fact that $\Phi$ is $\frac{1}{\sqrt{2\pi}}-$Lipschitz yields \begin{align*}
    \left|\Phi\left(\frac{-t}{2\sqrt{(2n+t)(1/4)}}\right) - \Phi\left(\frac{-t}{\sqrt{8 n (1/4)}}\right)\right| &\leq \frac{1}{\sqrt{2 \pi}}\left|\frac{-t}{\sqrt{(2n+t)}} - \frac{-t}{\sqrt{2n}}\right| \\
    &= \frac{1}{\sqrt{2 \pi}}\left|\frac{t^2}{\sqrt{2n(2n+t)}(\sqrt{2n + t} + \sqrt{2n})}\right| \\
    &\leq \frac{Ct^2}{n^{3/2}}.
\end{align*}The other term can be bounded in a similar way. Therefore, \begin{align*}
    \left|H_n (a) - \sum_{t=0}^{\lfloor a \sqrt{n}\rfloor} t \left[\Phi\left(\frac{-t}{\sqrt{2 n}}\right) + 1 - \Phi\left(\frac{t}{\sqrt{2n}}\right)\right]\right| &\leq C\sum_{t = 0}^{\lfloor a \sqrt{n}\rfloor } \left(\frac{t^3}{n^{3/2}} + tn^{-1/2} + \left|A_n - \frac{1}{4}\right|t\right) \\
    &\leq C\left(a^4 n^{1/2} + a^2 n^{1/2} + \left|A_n -\frac{1}{4}\right|a^2 n\right).
\end{align*}Therefore, for some constant $\Cr{Head_sum_gaussian_approx_constant}$ independent of $n$, we have \begin{align*}
    \frac{1}{n}\left|H_n (a) - \sum_{t=0}^{\lfloor a \sqrt{n}\rfloor} t \left[2\Phi\left(\frac{-t}{\sqrt{2 n}}\right)\right]\right| &\leq \Cr{Head_sum_gaussian_approx_constant} \left(a^4 n^{-1/2} + \left|A_n - \frac{1}{4}\right|a^2\right),
\end{align*}thereby proving \thref{Head_sum_gaussian_approx}. \hfill $\blacksquare$ \\

We can now complete our proof by combining the error estimates for $H_n (a)$ and $T_n (a)$ given above. We see that \begin{align*}
    \left|\frac{1}{n}\sum_{t=0}^{\infty}t \cdot P_n \left(|Z_1^{+}- n| > t\right) - 1\right|&= \left|\frac{1}{n}\left(H_n(a)+T_n (a)\right) - 1\right|\\
    &\leq  \left|\frac{1}{n}H_n (a) - 1\right| + \frac{1}{n}\left|T_n (a)\right| \\
    &= \mathcal{O}\left(a^4 n^{-1/2} + \left|A_n - \frac{1}{4}\right|a^2 + e^{-C' a}\right).
\end{align*} To finish our proof that $\lim_{n \to \infty}\mathbb{E}_n [(Z_1^{+}- n)^2] = 2$, we only need to choose an appropriate sequence $a_n$. In the case that $b_n \leq n^{-1/2}$, we can take $a(n) = C \log n$ to obtain a rate of convergence of $\mathcal{O}\left(\frac{\log^4 n}{\sqrt{n}}\right)$, because we have that \begin{align*}
    b_n \log^4 b_n \leq \frac{(\log n^{-1/2})^4}{\sqrt{n}} = \mathcal{O}\left(\frac{\log^4 n}{\sqrt{n}}\right).
\end{align*}On the other hand, if $b_n \geq n^{-1/2}$, we will take $a(n) = -\frac{1}{C'}\log\left(b_n\right)$ to obtain a rate of convergence of $\mathcal{O}\left(b_n \log^4 b_n\right)$. Finally, note that in this case $b_n \log^4 b_n$ is the dominant term, since \begin{align*}
    \frac{\log^4 n}{\sqrt{n}} = \frac{16 (\log n^{-1/2})^4}{\sqrt{n}} \leq 16 b_n (\log b_n)^4 = \mathcal{O}\left(b_n \log^4 b_n\right).
\end{align*}Combining the two cases, we see that $|\nu(n) - 2| = \mathcal{O}\left(n^{-1/2}\log^4 n + b_n \log^4 b_n\right)$. \end{proof}

\subsubsection{Rate of convergence for $\theta(n)$}
Since $\theta(n) = \frac{2\rho(n)}{\nu(n)}$, in light of the preceeding sections we have that \begin{align*}
    \theta =\lim_{n \to \infty} \theta(n) =  \lim_{n \to \infty} \frac{2 \rho(n)}{\nu(n)} = \delta.
\end{align*}More importantly, we can use the bounds we obtained on $|\rho(n) - \delta|$ and $|\nu(n) - 2|$ to bound the rate at which $\theta(n) \to \delta$. In particular, we will show that $|\theta(n) - \delta| = o\left(\frac{1}{\log n}\right)$. To do this, note that \begin{align*}
    \left|\theta(n) - \delta\right| &= \left|\frac{2\rho(n)}{\nu(n)} - \delta\right| \\ 
    &=\left|\frac{2\rho(n) - \rho(n)\nu(n) + \rho(n)\nu(n) - \nu(n)\delta}{\nu(n)}\right| \\
    &\leq \frac{\rho(n)}{\nu(n)}\left|\nu(n) - 2\right| + \left|\rho(n) - \delta\right|.
\end{align*}Under the assumption in \eqref{cookie_decay_assumption}, we have that 
\begin{align*}
    \left|\rho(n) - \delta\right| &= o\left(\frac{1}{\log n}\right).
\end{align*}
The term involving $|\nu(n) - 2|$ can be handled by using the bound from \thref{Periodic_cookies_nu_calculation}:\begin{align}
    |\nu(n) - 2| = \mathcal{O}\left(\frac{\log^4 n}{\sqrt{n}} + b_n \log^4 b_n\right). \label{bound_from_lemma}
\end{align}In the case that $b_n \leq n^{-1/2}$, the dominant term in \eqref{bound_from_lemma} will be $n^{-1/2}\log^4 n$, and then we will have that $|\nu(n) - 2| = \mathcal{O}\left(\frac{\log^4 n}{\sqrt{n}}\right) = o\left(\frac{1}{\log n}\right)$. On the other hand, if $b_n \geq n^{-1/2}$, the dominant term in \eqref{bound_from_lemma} will be $b_n \log^4 b_n$. We will now show that this bound is also $o\left(\frac{1}{\log n}\right)$. Since we assume that $\sum_{j=n}^{\infty} (2p_j - 1) = o\left(\frac{1}{\log n}\right)$, for all $\epsilon > 0$ there exists $n_{\epsilon}$ so that \begin{align*}
    \left|\sum_{j=n}^{\infty}(2p_j - 1)\right| &< \frac{\epsilon}{\log n}, \quad \forall\  n \geq n_{\epsilon}.
\end{align*}In turn, this implies that $|2p_n - 1| = o\left(\frac{1}{\log n}\right)$, since for any $\epsilon > 0$ \begin{align*}
    \left| 2p_n - 1 \right| &= \left|\sum_{j=n}^{\infty}(2p_j - 1) - \sum_{j=n-1}^{\infty}(2p_j - 1)\right| < \frac{2\epsilon}{\log n} \quad \forall \ n \geq n_{\epsilon},
\end{align*}and so $|2p_n - 1| = \mathcal{O}\left(\frac{1}{\log n}\right)$. Since $b_n = \frac{1}{4n}\sum_{j=1}^n (2p_j - 1)^2$, this implies that \begin{align*}
    b_n = \mathcal{O}\left(\frac{1}{\log^2 n}\right), 
\end{align*}which gives the following bound on $|\nu(n) - 2|$: \begin{align*}
    \left|\nu(n) - 2 \right| &= \mathcal{O}\left(\frac{\left(\log(\log n)\right)^4}{(\log n)^2}\right) = o\left(\frac{1}{\log n}\right).
\end{align*}Combining the bounds on $|\rho(n) - \delta|$ and $|\nu(n) - 2|$ shows that $|\theta(n) - \delta| = o\left(\frac{1}{\log n}\right)$.  
\section{Proof of \thref{Transient_example}: example of a transient ERW with $\delta = 1$}
In this section we prove \thref{Transient_example} by exhibiting a cookie environment where $\delta = 1$, but ERW is transient. We will build an environment where the first few cookies give a total drift greater than $1$, then fill the rest of the cookie stack with infinitely many weak negative cookies that are spaced very far apart (and the gaps filled with placebo cookies). Our guiding intuition in this construction is that we want the total drift $\delta$ contained in the cookie stack to be $1$, but only ``in the limit.'' For any finite number of visits to a given site, the total drift that the walker consumes at that site will be greater than $1$, and so we might expect that the walk behaves as though $\delta>1$, i.e. is transient to $+\infty$. To be explicit, consider the cookie environment $\boldsymbol{p} = \left(p_1 , p_2 , p_3 , \ldots\right)$, where \begin{align*}
    p_k &= \begin{cases} \frac{5}{6} & k = 1, 2 , 3 \\ \frac{1}{2} - \left(\frac{1}{2}\right)^{m+1} & k = 4^{4^m}, \ m = 1, 2, \ldots \\ \frac{1}{2} & \text{otherwise}\end{cases}.
\end{align*}Recall that $\nu = 2$, $\rho = \delta(\boldsymbol{p})$, and $\theta =\delta(\boldsymbol{p})$. In this case, \begin{align*}
    \delta &= \sum_{j=1}^{\infty} (2p_j - 1) = 2 + \sum_{m=1}^{\infty}-\left(\frac{1}{2}\right)^m = 1,
\end{align*} and so $\rho = \theta = 1$. Note that $\delta_m > 1$ when $m > 3$. To ensure that ERW in this environment is transient, we will check that the second part of \thref{Theorem_1.3_KOS_Periodic_Cookies} is satisfied. That is, we will need to verify that \begin{align*}
\theta (n) - 1 \geq \frac{2}{\log n}
\end{align*}for all $n$ sufficiently large. To that end, first observe that \begin{align*}
    \theta (n) - 1 = \frac{\rho(n)}{\nu (n)}\left(2-\nu(n)\right) + (\rho(n) - 1),
\end{align*}and so we will need suitable lower bounds on the right-hand side. The term involving $(2-\nu(n))$ will ultimately be negligible since there exists a constant $C$ such that \begin{align*}
    \left| A_n - \frac{1}{4}\right| = \frac{1}{4n}\sum_{k=1}^{n}(2p_k - 1)^2 \leq \frac{C}{n}
\end{align*}because the sum $\sum (2p_k - 1)^2$ converges. \thref{Periodic_cookies_nu_calculation} therefore implies that $|\nu(n) - 2| = \mathcal{O}\left(\frac{\log^4 n}{\sqrt{n}}\right)$. We now turn our attention to the term $(\rho(n) - 1)$. By \eqref{compute_rho}, $\rho(n) = \mathbb{E}[\delta_{T_n}]$, and so (assuming that $n > 3$) we have \begin{align}
    \rho(n) - 1 &= \mathbb{E}[\delta_{T_n}] - 1 \nonumber \\
    &\stackrel{n > 3}{=} \mathbb{E}\left[2 + \sum_{k=4}^{T_n}(2p_k - 1)\right] - 1 \nonumber \\
    &=1 +  \mathbb{E}\left[\sum_{k=4}^{T_n}(2p_k - 1)\right]. \label{example_rho_line}
\end{align}We will now find a formula for the expectation in \eqref{example_rho_line}. To this end, let $C_{\boldsymbol{p}}(x) = \#\{j \leq x \ : \ p_j < 1/2\}$ be the number of negative drift cookies in $\boldsymbol{p}$ up to cookie $x$. This allows us to rewrite the expression in \eqref{example_rho_line}: \begin{align*}
    1 + \mathbb{E}\left[\sum_{k=4}^{T_n}(2p_k - 1)\right] &= 1 + \mathbb{E}\left[\sum_{m=1}^{C_{\boldsymbol{p}}(T_n)}-\left(\frac{1}{2}\right)^m\right] \\
    &= \mathbb{E}\left[\left(\frac{1}{2}\right)^{C_{\boldsymbol{p}}(T_n)}\right].
\end{align*}Now, we will obtain a lower bound on this quantity. \begin{align*}
    \mathbb{E}\left[\left(\frac{1}{2}\right)^{C_{\boldsymbol{p}}(T_n)}\right] &\geq \left(\frac{1}{2}\right)^{\log_4 (\log n)} P\left(C_{\boldsymbol{p}}(T_n) \leq \log_{4}(\log n)\right) \\
    &= \frac{1}{\sqrt{\log n}} P\left(C_{\boldsymbol{p}}(T_n) \leq \log_{4}(\log n)\right).
\end{align*}Our choice of environment guarantees that the number of negative drift cookies in the first, say, $3n$ cookies is no more than $\log_{4}(\log n)$. This fact together with the concentration bound for $T_n$ in \eqref{T_n_concentration} gives
\begin{align*}
    \frac{1}{\sqrt{\log n}}P\left(C_{\boldsymbol{p}}(T_n) \leq \log_4 (\log n)\right) &\geq \frac{1}{\sqrt{\log n}}P\left(T_n \leq 3n \right) \\
    &\geq \frac{1 - e^{-C n}}{\sqrt{\log n}}.
\end{align*}Combining this inequality with the bound for $\nu(n) - 2$ discussed above, we have for all sufficiently large $n$ that \begin{align*}
    \theta(n) - 1 &\geq \frac{1 - e^{-C n}}{\sqrt{\log n}} + \mathcal{O}\left(\frac{\log^4 n}{\sqrt{n}}\right) \\
    &\geq \frac{2}{\log n}.
\end{align*}Therefore, the FBLP associated to the ERW in $\boldsymbol{p}$ has a positive probability of survival by \thref{Theorem_1.3_KOS_Periodic_Cookies}, and so the ERW in $\boldsymbol{p}$ is a.s. transient to $+\infty$. \hfill $\blacksquare$

\section{Appendix}

\begin{lemma}\label{Walds_Identity}\thlabel{Wald} Let $Z_1^{+}$ denote the number of offspring in the first generation of the forward branching-like process, and let $T_n$ be the trial on which the $n$th failure in the coin tosses that determine the value of $Z_1^{+}$ occurs (where the $i$th coin comes up heads with probability $p_i$). Then \begin{align*}
    \mathbb{E}_n [Z_1^{+}] = \mathbb{E}\left[\sum_{j=1}^{T_n}\xi_j\right] = \mathbb{E}\left[\sum_{j=1}^{T_n} p_j \right].
\end{align*}
\end{lemma}

\begin{proof}
The statement is essentially a version of Wald's identity, and is proved in a similar way. First, we write \begin{align}
    \mathbb{E}\left[\sum_{j=1}^{T_n} \xi_k\right] = \sum_{j=1}^{\infty}\mathbb{E}\left[\xi_j \boldsymbol{1}_{\{T_n \geq j\}}\right]. \label{Wald_proof_line}
\end{align}Because the event $\{T_n \geq j\} = \{T_n > j-1\}$ depends only on $\xi_1 , \ldots, \xi_j$, we have $\mathbb{E}\left[\xi_j \boldsymbol{1}_{\{T_n \geq j\}}\right] = p_j P(T_n \geq j)$. Using this fact, we can rewrite the right-hand side of \eqref{Wald_proof_line} as \begin{align*}
    \sum_{j=1}^{\infty} p_j P(T_n \geq j) = \sum_{j=1}^{\infty}p_j \mathbb{E}[\boldsymbol{1}_{\{T_n \geq j\}}] = \mathbb{E}\left[\sum_{j=1}^{\infty}p_j \boldsymbol{1}_{\{T_n \geq j}\}\right] = \mathbb{E}\left[\sum_{j=1}^{T_n}p_j\right],
    \end{align*}thereby proving the claim.
\end{proof}

\clearpage
\bibliography{references}{}
\bibliographystyle{amsplain}

\end{document}

%% file: preamble.tex
\usepackage{amssymb,amsmath,mathtools,mathrsfs,setspace}
\usepackage{amsthm}
\usepackage{pgf,tikz}
\usepackage{skull}
\usepackage{enumerate}
\usepackage[UKenglish]{datetime} 
\newdateformat{UKvardate}{
\THEDAY\ \monthname[\THEMONTH] \THEYEAR}
\UKvardate

\usepackage[margin=1in]{geometry}

\newtheorem{theorem}{Theorem}[section]
\newtheorem{lemma}{Lemma}